\documentclass[12pt,reqno,oneside]{amsart}
\usepackage{amsfonts}
\usepackage{amsmath, amssymb}
\usepackage{hyperref}

\numberwithin{equation}{section} 
\newtheorem{theorem}{Theorem}[section]
\newtheorem{corollary}{Corollary}[section]

\theoremstyle{definition}

\theoremstyle{remark}
\newtheorem{remark}{Remark}[section]
\numberwithin{equation}{section}

\input{tcilatex}

\begin{document}
\title[Certain Classes of Bi-univalent functions defined by Horadam Polynomials]{Initial bounds for certain classes  of bi-univalent functions defined by Horadam Polynomials}
\author{Nanjundan Magesh, Jagadeesan Yamini and Chinnaswamy Abirami}
\maketitle
\vspace{-.5cm}
\begin{center}
	Post-Graduate and Research Department of Mathematics\\	Government Arts College for Men, Krishnagiri 635001, Tamilnadu, India.\\
	{\bf e-mail~:~}\verb+nmagi_2000@yahoo.co.in+
\end{center}
\begin{center}
	Department of Mathematics, Govt First Grade College Vijayanagar, Bangalore-560104, Karnataka, India.\\
	{\bf e-mail~:~}\verb+yaminibalaji@gmail.com+
\end{center}
\begin{center}
	Faculty of Engineering and Technology\\ SRM University, Kattankulathur-603203, Tamilnadu, India.\\
	{\bf e-mail~:~}\verb+shreelekha07@yahoo.com+
\end{center}
\begin{abstract}
	Our present investigation is motivated essentially by the fact that, in Geometric Function Theory, one can find many interesting and fruitful usages of a wide variety of special functions and special polynomials. The main purpose of this article is to make use of the Horadam polynomials $h_{n}(x)$ and the generating function $\Pi(x,\; z)$, in order to introduce 	three new subclasses of the bi-univalent function class $\Sigma.$ For functions belonging to the defined classes, we then derive coefficient inequalities and the Fekete–Szeg\"{o} inequalities. Some interesting observations of the results presented here are also discussed. We also provide relevant connections of our results with those considered in earlier investigations. 
\
\\
\
\\
\
{\it keywords and Phrases~:~ Univalent functions, bi-univalent functions, bi-Mocanu-convex functions, bi-$\alpha-$starlike functions, bi-starlike functions, bi-convex functions, Fekete-Szeg\"{o} problem, Chebyshev polynomials, Horadam polynomials.}
\
\\
\
\\
\
{\bf 2010 Mathematics Subject Classifications~:~ Primary 11B 39,\; 30C45,\; 33C45; Secondary 30C50,\; 33C05.}
\end{abstract}

\section{Introduction}
Let $\mathbb{R}=\left( -\infty ,\infty \right) $\ be the set of real
numbers, $\mathbb{C}$ be the set of complex numbers and%
\begin{equation*}
\mathbb{N}:=\left\{ 1,2,3,\ldots \right\} =\mathbb{N}_{0}\backslash \left\{
0\right\}
\end{equation*}%
be the set of positive integers. Let $\mathcal{A}$ denote the class of functions of the form
\begin{equation}
f(z)=z+\sum\limits_{n=2}^{\infty }a_{n}z^{n}  \label{Int-e1}
\end{equation}%
which are analytic in the open unit disk $\Delta=\{z:z\in \mathbb{C}\,\,%
\mathrm{and}\,\,|z|<1\}.$ Further, by $\mathcal{S}$ we shall denote the
class of all functions in $\mathcal{A}$ which are univalent in $\Delta.$

It is well known that every function $f\in \mathcal{S}$ has an inverse $%
f^{-1},$ defined by
\begin{equation*}
f^{-1}(f(z))=z \qquad (z\in\Delta)
\end{equation*}
and
\begin{equation*}
f(f^{-1}(w))=w \qquad (|w| < r_0(f);\,\, r_0(f) \geqq \frac{1}{4}),
\end{equation*}
where
\begin{equation*}
f^{-1}(w) = w - a_2w^2 + (2a_2^2-a_3)w^3 - (5a_2^3-5a_2a_3+a_4)w^4+\ldots .
\end{equation*}

A function $f\in \mathcal{A}$ is said to be bi-univalent in $\Delta$ if
both a function $f$ and it's inverse $f^{-1}$ are univalent in $\Delta.$ Let $\Sigma $
denote the class of bi-univalent functions in $\Delta$ given by (\ref{Int-e1}). 

In 2010, Srivastava et al. \cite{HMS-AKM-PG} revived the study of bi-univalent functions by their pioneering work on the study of coefficient problems. Various subclasses of
the bi-univalent function class $\Sigma $ were introduced and non-sharp
estimates on the first two coefficients $|a_{2}|$ and $|a_{3}|$ in the
Taylor-Maclaurin series expansion (\ref{Int-e1}) were found in the
recent investigations (see, for example, \cite{Ali-Ravi-Ma-Mina-class,SA-SY-2017-GJM,SA-SY-2017-arXiv,MC-ED-HMS-2018-TJM,VBG-SBJ-Ganita-2018,HOG-GMS-JS-Fib-2018,Jay-SGH-SAH-2014,SK-EAA-AZ-2017-MJM,SKL-VR-SS-2014-AAA,XL-APW-2012-IMF,NM-SB-2018-AM,HO-NM-VKB-2019-AEJM,ZP-QH-2014-AMS,HMS-SA-SY-2018-IJSTTS,HMS-FMS-HOG-2018-Filomat,HMS-SSE-SGH-JMJ-2018-IMS,HMS-NM-JY-2014,ZT-LX-2018-JMI,LX-XL-2015-Filomat,PZ-2014-BBMSSS}) 
and including the references therein. The afore-cited all these papers on the subject were actually motivated by the work of Srivastava et al. \cite{HMS-AKM-PG}. However, the problem to find the coefficient bounds on $|a_{n}|$ ($n=3,4,\dots $) for functions $f\in \Sigma $ is still an open problem.

For analytic functions $f$ and $g$ in $\Delta,$ $f$ is said to be
subordinate to $g$ if there exists an analytic function $w$ such that 
\begin{equation*}
w(0)=0, \quad \quad |w(z)|<1 \quad \mathrm{and} \quad f(z)=g(w(z)) \qquad
(z\in\Delta).
\end{equation*}
This subordination will be denoted here by
\begin{equation*}
f \prec g \qquad (z\in\Delta)
\end{equation*}
or, conventionally, by
\begin{equation*}
f(z) \prec g(z) \qquad (z\in\Delta).
\end{equation*}

In particular, when $g$ is univalent in $\Delta,$
\begin{equation*}
f \prec g \qquad (z\in\Delta) ~\Leftrightarrow ~f(0)=g(0) \quad \mathrm{%
	and} \quad f(\Delta) \subset g(\Delta).
\end{equation*}

The Horadam polynomials $h_{n}(x,\; a,\; b;\; p,\; q),$ or briefly  $h_{n}(x)$ are given by the following recurrence relation (see \cite{AFH-JMM-1985-FQ,TH-EGK-2009-IMF})): 
\begin{eqnarray}\label{HP-RR}
h_{n}(x) &=& pxh_{n-1}(x)+qh_{n-2}(x) \qquad (n\in \mathbb{N})
\end{eqnarray}
with 
\begin{eqnarray}\label{HP-ab}
h_{1}(x)=a \qquad \hbox{and} \qquad h_{2}(x)=bx
\end{eqnarray}
for some real constants $a,$ $b,$ $p$ and $q.$ 

The generating function of the Horadam polynomials $h_{n}(x)$ (see \cite{TH-EGK-2009-IMF}) is given by 
	\begin{equation}\label{HP-GF}
	\Pi(x,\; z):= \sum\limits_{n=1}^{\infty}h_{n}(x)z^{n-1} = \dfrac{a+(b-ap)xz}{1-pxz-qz^{2}}~.
	\end{equation}
Here, and in what follows, the argument $x \in \mathbb{R}$
	is independent of the argument $z \in \mathbb{C};$ that is, $x \neq \Re(z).$

Note that for particular values of $a,$ $b,$ $p$ and $q$, the Horadam polynomial $h_{n}(x)$ leads to various polynomials, among those, we list few cases here (see,  \cite{AFH-JMM-1985-FQ,TH-EGK-2009-IMF} for more details):  
\begin{enumerate}
	\item For $a=b=p=q=1,$ we have the Fibonacci polynomials $F_{n}(x)$. 
	\item For $a=2$ and $b=p=q=1,$ we obtain the Lucas polynomials $L_{n}(x)$. 
	\item For $a=q=1$ and $b=p=2,$ we get the Pell polynomials $P_{n}(x)$. 
	\item For $a=b=p=2$ and $q=1,$ we attain the Pell-Lucas polynomials $Q_{n}(x)$. 
	\item For $a=b=1,\; p=2$ and $q=-1,$ we have the Chebyshev polynomials $T_{n}(x)$ of the first kind 
	\item For $a=1,\; b=p=2$ and $q=-1,$ we obtain the Chebyshev polynomials $U_{n}(x)$ of the second kind. 

\end{enumerate}

Recently, in literature, the coefficient estimates are found for functions in the class of univalent and bi-univalent functions associated with certain polynomials like the Faber polynomial \cite{Jay-SGH-SAH-2014}, the Chebyshev polynomials \cite{VBG-SBJ-Ganita-2018}, the Horadam polynomial \cite{HMS-SA-SY-2018-IJSTTS}. Motivated in these lines, estimates on initial coefficients of the Taylor-Maclaurin series expansion \eqref{Int-e1} and Fekete-Szeg\"{o} inequalities for certain classes of bi-univalent functions defined by means of Horadam polynomials are obtained. The classes introduced in this paper are motivated by the corresponding classes investigated in \cite{Ali-Ravi-Ma-Mina-class,SKL-VR-SS-2014-AAA,HMS-SA-SY-2018-IJSTTS,ZP-QH-2014-AMS}.
\section{Coefficient Estimates and Fekete-Szeg\"{o} Inequalities}
A function $f\in \Sigma$ of the form \eqref{Int-e1} belongs to the class $\mathcal{S}_{\Sigma}^{\ast}(\alpha,\; x),$ $\alpha \geq 0$  and $z,\; w \in \Delta,$ if the following conditions are satisfied: 
\begin{equation*} \label{CR-RAP-NM-P1-e1}
\dfrac{zf'(z)}{f(z)}+\alpha \dfrac{z^{2}f''(z)}{f(z)} \prec \Pi(x,\; z)+1-a
\end{equation*}
and for $g(w)=f^{-1}(w)$ 
\begin{equation*}  \label{CR-RAP-NM-P1-e2}
\dfrac{wg'(w)}{g(w)}+\alpha\dfrac{w^{2}g''(w)}{g(w)} \prec \Pi(x,\; w) +1-a,
\end{equation*}
where the real constants $a$ and $b$ are as in \eqref{HP-ab}.

Note that $S_{\Sigma}^{\ast}(x) \equiv \mathcal{S}_{\Sigma}^{\ast}(0,\; x)$ was introduced and studied by  Srivastava et al. \cite{HMS-SA-SY-2018-IJSTTS}.	 

\begin{remark}\label{Rem-Positive}
If $a=p=x=1,$ $b=2$ and $q=0,$ then we have
\[
\dfrac{zf'(z)}{f(z)}+\alpha \dfrac{z^{2}f''(z)}{f(z)} \prec \dfrac{1+z}{1-z} \qquad (z\in \Delta).
\]

In this case, the function $f$ maps the open unit disk $\Delta$ onto
the half-plane given by $\Re \left(\tfrac{+z}{-z}\right) > 0,$ since the expression $\tfrac{zf'(z)}{f(z)}+\alpha \tfrac{z^{2}f''(z)}{f(z)}$
takes on values in the half-plane. If, on the other hand, we
restrict our considerations for a given univalent function
$p(z) \in \Delta,$ we can investigate the corresponding mapping
problems for other regions of the complex $z-$plane instead
of the half-plane $\Re(z)>0$. In this way, one can introduce many other subclasses of the function class $\mathcal{S}_{\Sigma}^{\ast}(\alpha,\; x)$.
\end{remark}

\begin{remark}\label{Rem-Che-P}
	When $a=1,$ $b=p=2,$ $q=-1$ and $x\rightarrow t$, the generating function in \eqref{HP-GF} reduces to that of the Chebyshev polynomial $U_{n}(t)$ of the second kind, which is given explicitly by 
	\[
	U_{n}(t) = (n+1){}_{2}F_{1}\left(-n,\; n+2;\; \tfrac{3}{2};\; \tfrac{1-t}{2}\right) = \frac{\sin(n+1)\varphi}{\sin\varphi}, \qquad (t=\cos \varphi)
	\]
	in terms of the hypergeometric function ${}_{2}F_{1}.$ 
\end{remark}

In view of Remark \ref{Rem-Che-P}, the bi-univalent function class  $\mathcal{S}_{\Sigma}^{\ast}(x)$ reduces to $\mathcal{S}_{\Sigma}^{\ast}(t)$ and this class was studied earlier in \cite{SA-SY-2017-GJM,NM-SB-2018-AM}. For functions in the class $\mathcal{S}_{\Sigma}^{\ast}(\alpha,\; x)$, the following coefficient estimates and Fekete-Szeg\"{o} inequality are obtained. 
\begin{theorem}
	\label{Th3} Let $f(z)=z+\sum\limits_{n=2}^{\infty}a_{n}z^{n}$ be in the
	class $\mathcal{S}_{\Sigma}^{\ast}(\alpha,\; x)$. Then 
	\begin{align*}
	\left|a_{2}\right| &\leq \dfrac{\left|bx\right|\sqrt{\left|bx\right|}}{\sqrt{\left|[\left(1+4\alpha\right)b-p(1+2\alpha)^{2}]bx^{2}-qa(1+2\alpha)^{2}\right|}},
	\qquad \hbox{and} \qquad
	\left|a_{3}\right|\leq \dfrac{\left|bx\right|}{2+6\alpha}  + \dfrac{b^{2}x^{2}}{(1+2\alpha)^{2}}
	\end{align*}
	and for $\nu \in \mathbb{R}$
	\begin{eqnarray*}
		\left|a_{3}-\nu a_{2}^{2}\right|
		\leq \left\{
		\begin{array}{ll}
			\dfrac{\left|bx\right|}{2+6\alpha} &\\
			\qquad \hbox{if} \qquad
			\left|\nu - 1\right|\leq \dfrac{\left|[\left(1+4\alpha\right)b-p(1+2\alpha)^{2}]bx^{2}-qa(1+2\alpha)^{2}\right|}{2b^{2}x^{2}\left(1+3\alpha\right)}\\\\
			\dfrac{\left|bx\right|^{3}\left|\nu - 1\right|}{\left|[\left(1+4\alpha\right)b-p(1+2\alpha)^{2}]bx^{2}-qa(1+2\alpha)^{2}\right|} &\\
			\qquad \hbox{if} \qquad 
			\left|\nu - 1\right|\geq \dfrac{\left|[\left(1+4\alpha\right)b-p(1+2\alpha)^{2}]bx^{2}-qa(1+2\alpha)^{2}\right|}{2b^{2}x^{2}\left(1+3\alpha\right)}.
		\end{array}		
		\right. 
	\end{eqnarray*}
\end{theorem}

\begin{proof}
	Let $f\in\mathcal{S}_{\Sigma}^{\ast}(\alpha,\; x)$ be given by Taylor-Maclaurin expansion \eqref{Int-e1}. Then, there are analytic functions $u$ and $v$ such that 
	\[
	u(0)=0;\quad v(0)=0,\quad \left|u(z)\right|<1 \quad \hbox{and} \quad \left|v(z)\right|<1 \quad (\forall\; z,\; w \in \Delta),
	\]
	we can write
	\begin{equation}  \label{4.2}
	\dfrac{zf'(z)}{f(z)}+\alpha \dfrac{z^{2}f''(z)}{f(z)}
	= \Pi(x,\; u(z))+1-a
	\end{equation}
	and 
	\begin{equation}  \label{4.3}
	\dfrac{wg'(w)}{g(w)}+\alpha\dfrac{w^{2}g''(w)}{g(w)}
	=\Pi(x,\; v(w))+1-a.
	\end{equation}
	Or, equivalently,
	\begin{eqnarray}  \label{4.2a}
	\dfrac{zf'(z)}{f(z)}+\alpha \dfrac{z^{2}f''(z)}{f(z)}
	= 1+h_{1}(x)-a+h_{2}(x)u(z)+h_{3}(x)[u(z)]^{2}+\ldots
	\end{eqnarray}
	and 
	\begin{eqnarray}  \label{4.3a}
	\dfrac{wg'(w)}{g(w)}+\alpha\dfrac{w^{2}g''(w)}{g(w)}
	= 1+h_{1}(x)-a+h_{2}(x)v(w)+h_{3}(x)[v(w)]^{2}+\ldots.
	\end{eqnarray}
	From  \eqref{4.2a} and \eqref{4.3a}, we obtain
	\begin{eqnarray}  \label{4.2b}
	\dfrac{zf'(z)}{f(z)}+\alpha \dfrac{z^{2}f''(z)}{f(z)}= 1+h_{2}(x)u_{1}z+[h_{2}(x)u_{2}+h_{3}(x)u_{1}^{2}]z^{2}+\ldots
	\end{eqnarray}
	and 
	\begin{eqnarray}  \label{4.3b}
	\dfrac{wg'(w)}{g(w)}+\alpha\dfrac{w^{2}g''(w)}{g(w)}
	= 1+h_{2}(x)v_{1}w+[h_{2}(x)v_{2}+h_{3}(x)v_{1}^{2}]w^{2}+\ldots .
	\end{eqnarray}
	It is fairly well known that 
	\[
	\left|u(z)\right| = \left|u_{1}z+u_{2}z^{2}+\ldots\right| <1 \qquad \hbox{and}\qquad 
	\left|v(z)\right| = \left|v_{1}w+v_{2}w^{2}+\ldots\right| <1,
	\]
	then
	\[
	\left|u_{k}\right| \leq 1 \qquad \hbox{and} \qquad \left|v_{k}\right| \leq 1 \qquad (k \in \mathbb{N}).
	\]
	Thus upon comparing the corresponding coefficients in \eqref{4.2b} and \eqref{4.3b}, we have 
	\begin{equation}  \label{4.6}
	\left(1+2\alpha\right)a_{2}
	= h_{2}(x)u_{1}
	\end{equation}
	\begin{equation}  \label{4.7}
	2\left(1+3\alpha\right)a_{3} - \left(1+2\alpha\right)a_{2}^{2}
	= h_{2}(x)u_{2}+h_{3}(x)u_{1}^{2}
	\end{equation}
	
	\begin{equation}  \label{4.8}
	-\left(1+2\alpha\right)a_{2}
	=h_{2}(x)v_{1}
	\end{equation}
	
	and 
	\begin{equation}  \label{4.9}
	\left(3+10\alpha\right)a_{2}^{2} - 2\left(1+3\alpha\right)a_{3}
	=h_{2}(x)v_{2}+h_{3}(x)v_{1}^{2}.
	\end{equation}
	
	From \eqref{4.6} and \eqref{4.8}, we can easily see that
	\begin{equation} \label{4.9a}
	u_{1}=-v_{1}
	\end{equation}
	and
	\begin{eqnarray} 
	2(1+2\alpha)^{2} a_{2}^{2} &=& [h_{2}(x)]^{2}(u_{1}^{2} + v_{1}^{2})\nonumber\\
	a_{2}^2 &=& \dfrac{[h_{2}(x)]^{2}(u_{1}^{2} + v_{1}^{2})}{2(1+2\alpha)^{2}}~.\label{4.10}
	\end{eqnarray}
	If we add   \eqref{4.7} to \eqref{4.9}, we get
	\begin{eqnarray}
	2\left(1+4\alpha\right) a_{2}^{2}
	= h_{2}(x)(u_{2}+v_{2}) + h_{3}(x)(u_{1}^{2}+v_{1}^{2}).\label{4.11}
	\end{eqnarray}
	By substituting  \eqref{4.10} in \eqref{4.11}, we reduce that
	\begin{eqnarray}
	a_{2}^{2} &=& \dfrac{[h_{2}(x)]^{3}\left(u_{2}+v_{2}\right)}{2\left(1+4\alpha\right)[h_{2}(x)]^{2}-2h_{3}(x)(1+2\alpha)^{2}}\label{4.12}
	\end{eqnarray}
	which yields
	\begin{eqnarray}
	\left|a_{2}\right| &\leq& \dfrac{\left|bx\right|\sqrt{\left|bx\right|}}{\sqrt{\left|[\left(1+4\alpha\right)b-p(1+2\alpha)^{2}]bx^{2}-qa(1+2\alpha)^{2}\right|}}.\label{4.13}
	\end{eqnarray}
	By subtracting \eqref{4.9} from \eqref{4.7} and in view of \eqref{4.9a} , we obtain
	\begin{eqnarray}
	4(1+3\alpha)a_{3} - 4(1+3\alpha) a_{2}^{2} &=& h_{2}(x)\left(u_{2}-v_{2}\right) + h_{3}(x)\left(u_{1}^{2}-v_{1}^{2}\right)\nonumber\\
	a_{3}&=& \dfrac{h_{2}(x)\left(u_{2}-v_{2}\right)}{4(1+3\alpha)}  + a_{2}^{2}.\label{4.14}
	\end{eqnarray}
	Then in view of \eqref{4.10}, \eqref{4.14} becomes
	\begin{eqnarray*}
		a_{3}&=& \dfrac{h_{2}(x)\left(u_{2}-v_{2}\right)}{4(1+3\alpha)}  + \dfrac{[h_{2}(x)]^{2}(u_{1}^{2} + v_{1}^{2})}{2(1+2\alpha)^{2}}.\label{4.15}
	\end{eqnarray*}
	Applying \eqref{HP-ab}, we deduce that 
	\begin{eqnarray*}
		\left|a_{3}\right|&\leq& \dfrac{\left|bx\right|}{2+6\alpha}  + \dfrac{b^{2}x^{2}}{(1+2\alpha)^{2}}.\label{4.16}
	\end{eqnarray*}
	From \eqref{4.14}, for $\nu\in \mathbb{R},$ we write
	\begin{eqnarray}\label{Th1-Fekete-e1}
	a_{3} -\nu a_{2}^{2} = \dfrac{h_{2}(x)\left(u_{2}-v_{2}\right)}{4(1+3\alpha)} + \left(1-\nu\right)a_{2}^{2}.
	\end{eqnarray}
	By substituting \eqref{4.12} in \eqref{Th1-Fekete-e1}, we have 
	\begin{eqnarray}\label{Th1-Fekete-e2}
	a_{3} -\nu a_{2}^{2} 
	&=& \dfrac{h_{2}(x)\left(u_{2}-v_{2}\right)}{4(1+3\alpha)}+  \left(\dfrac{\left(1-\nu\right)[h_{2}(x)]^{3}\left(u_{2}+v_{2}\right)}{2[\left(1+4\alpha\right)[h_{2}(x)]^{2}-h_{3}(x)(1+2\alpha)^{2}]}\right)\nonumber\\
	&=& h_{2}(x)\left\{\left(\Omega(\nu,\; x) + \dfrac{1}{4\left(1+3\alpha\right)}\right)u_{2}
	+ \left(\Omega(\nu,\; x) - \dfrac{1}{4\left(1+3\alpha\right)}\right)v_{2}\right\},\nonumber\\
	\end{eqnarray}
	where
	\begin{eqnarray*}
		\Omega(\nu,\; x) = \dfrac{\left(1-\nu\right)[h_{2}(x)]^{2}}{2\left(1+4\alpha\right)[h_{2}(x)]^{2}-2h_{3}(x)(1+2\alpha)^{2}}.
	\end{eqnarray*}
	Hence, in view of \eqref{HP-ab}, we conclude that 
	\begin{eqnarray*}
		\left|a_{3}-\nu a_{2}^{2}\right|
		\leq \left\{
		\begin{array}{ll}
			\dfrac{\left|h_{2}(x)\right|}{2+6\alpha} &; 
			0 \leq \left|\Omega(\nu,\; x)\right|\leq \dfrac{1}{4\left(1+3\alpha\right)}\\\\
			2\left|h_{2}(x)\right|\left|\Omega(\nu,\; x)\right| &; 
			\left|\Omega(\nu,\; x)\right|\geq \dfrac{1}{4\left(1+3\alpha\right)}
		\end{array}		
		\right.
	\end{eqnarray*}
	which evidently completes the proof of Theorem \ref{Th1}.
\end{proof}
For $\alpha=0,$ Theorem \ref{Th3} readily yields the following coefficient estimates for $\mathcal{S}_{\Sigma}^{\ast}(x)$.
\begin{corollary}
	\label{cor2.1} Let $f(z)=z+\sum\limits_{n=2}^{\infty}a_{n}z^{n}$ be in the
	class $\mathcal{S}_{\Sigma}^{\ast}(x)$. Then 
	\begin{align*}
	\left|a_{2}\right| &\leq \dfrac{\left|bx\right|\sqrt{\left|bx\right|}}{\sqrt{\left|[b-p]bx^{2}-qa\right|}},
	\qquad \hbox{and} \qquad
	\left|a_{3}\right|\leq \dfrac{\left|bx\right|}{2} + b^{2}x^{2}
	\end{align*}
	and for $\nu \in \mathbb{R}$
	\begin{eqnarray*}
		\left|a_{3}-\nu a_{2}^{2}\right|
		\leq \left\{
		\begin{array}{ll}
			\dfrac{\left|bx\right|}{2} & 
			~\hbox{if}~\left|\nu - 1\right|\leq \dfrac{\left|[b-p]bx^{2}-qa\right|}{2b^{2}x^{2}}\\\\
			\dfrac{\left|bx\right|^{3}\left|\nu - 1\right|}{\left|[b-p]bx^{2}-qa\right|} &
			~\hbox{if}~\left|\nu - 1\right|\geq \dfrac{\left|[b-p]bx^{2}-qa\right|}{2b^{2}x^{2}}.
		\end{array}		
		\right. 
	\end{eqnarray*}
\end{corollary}
In view of Remark \ref{Rem-Che-P}, Theorem \ref{Th3} can be shown to yield
the following result.
\begin{corollary}
	\label{Th3-Cor} Let $f(z)=z+\sum\limits_{n=2}^{\infty}a_{n}z^{n}$ be in the
	class $\mathcal{S}_{\Sigma}^{\ast}(\alpha,\; t)$. Then 
	\begin{align*}
	\left|a_{2}\right| &\leq \dfrac{2t\sqrt{2t}}{\sqrt{\left|(1+2\alpha)^{2}-16\alpha^{2}t^{2}\right|}},
	\qquad \hbox{and} \qquad
	\left|a_{3}\right|\leq \dfrac{t}{1+3\alpha}  + \dfrac{4t^{2}}{(1+2\alpha)^{2}}
	\end{align*}
	and for $\nu \in \mathbb{R}$
	\begin{eqnarray*}
		\left|a_{3}-\nu a_{2}^{2}\right|
		\leq \left\{
		\begin{array}{ll}
			\dfrac{2t}{2+6\alpha} &~\hbox{if}~
			\left|\nu - 1\right|\leq \dfrac{\left|(1+2\alpha)^{2}-16\alpha^{2}t^{2}\right|}{8t^{2}\left(1+3\alpha\right)}\\\\
			\dfrac{8t^{3}\left|\nu - 1\right|}{\left|(1+2\alpha)^{2}-16\alpha^{2}t^{2}\right|} &~\hbox{if}~
			\left|\nu - 1\right|\geq \dfrac{\left|(1+2\alpha)^{2}-16\alpha^{2}t^{2}\right|}{8t^{2}\left(1+3\alpha\right)}.
		\end{array}		
		\right. 
	\end{eqnarray*}
\end{corollary}
\begin{remark} Results obtained in Corollary \ref{cor2.1} coincide with results obtained in \cite{HMS-SA-SY-2018-IJSTTS}. For $\alpha=0$, Corollary \ref{Th3-Cor} reduces to the results discussed in \cite{SA-SY-2017-GJM,NM-SB-2018-AM}. 
\end{remark}
Next, a function $f\in \Sigma$ of the form \eqref{Int-e1} belongs to the class $\mathcal{M}_{\Sigma}(\alpha,\; x),
$ $0\leqq \alpha \leqq 1$  and $z,\; w \in \Delta,$ if the following conditions are satisfied: 
\begin{equation*} \label{CR-RAP-NM-P1-e1}
(1-\alpha)\dfrac{zf'(z)}{f(z)}+
\alpha\left(1+\dfrac{zf''(z)}{f'(z)}\right) \prec \Pi(x,\; z)+1-a
\end{equation*}
and for $g(w)=f^{-1}(w)$ 
\begin{equation*}  \label{CR-RAP-NM-P1-e2}
(1-\alpha)\dfrac{wg'(w)}{g(w)}+
\alpha\left(1+\dfrac{wg''(w)}{g'(w)}\right) \prec \Pi(x,\; w) +1-a,
\end{equation*}
where the real constants $a$ and $b$ are as in \eqref{HP-ab}.

Note that the class $\mathcal{M}_{\Sigma}(\alpha,\; x),$ unifies the classes $S_{\Sigma}^{\ast}(x)$ and $K_{\Sigma}(x)$ like $\mathcal{M}_{\Sigma}(0,\; x) \equiv S_{\Sigma}^{\ast}(x) $ and $\mathcal{M}_{\Sigma}(1,\; x) \equiv K_{\Sigma}(x)$. In view of Remark \ref{Rem-Positive}, similarly, one can define many subclasses for the expression $(1-\alpha)\tfrac{zf'(z)}{f(z)}+\alpha\left(1+\tfrac{zf''(z)}{f'(z)}\right)$.
In view of Remark \ref{Rem-Che-P}, the bi-univalent function classes  $\mathcal{M}_{\Sigma}^{\ast}(\alpha,\; x)$ would become the class $\mathcal{M}_{\Sigma}^{\ast}(\alpha,\; t)$ and the class $\mathcal{M}_{\Sigma}^{\ast}(\alpha,\; t)$ introduced and studied by Alt\i nkaya and Yal\c cin \cite{SA-SY-2017-arXiv}. For functions in the class $\mathcal{M}_{\Sigma}(\alpha,\; x)$, the following coefficient estimates and Fekete-Szeg\"{o} inequality are obtained. 
\begin{theorem}
	\label{Th1} Let $f(z)=z+\sum\limits_{n=2}^{\infty}a_{n}z^{n}$ be in the
	class $\mathcal{M}_{\Sigma}(\alpha,\; x)$. Then 
	\begin{align*}
	\left|a_{2}\right| &\leq \dfrac{\left|bx\right|\sqrt{\left|bx\right|}}{\sqrt{\left|[\left(1+\alpha\right)b-p(1+\alpha)^{2}]bx^{2}-qa(1+\alpha)^{2}\right|}},
	\qquad \hbox{and} \qquad
	\left|a_{3}\right|\leq \dfrac{\left|bx\right|}{2+4\alpha}  + \dfrac{b^{2}x^{2}}{(1+\alpha)^{2}}
	\end{align*}
	and for $\nu \in \mathbb{R}$
	\begin{eqnarray*}
		\left|a_{3}-\nu a_{2}^{2}\right|
		\leq \left\{
		\begin{array}{ll}
			\dfrac{\left|bx\right|}{2+4\alpha} &\\
			\qquad\hbox{if}\qquad
			\left|\nu - 1\right|\leq \dfrac{\left|[\left(1+\alpha\right)b-p(1+\alpha)^{2}]bx^{2}-qa(1+\alpha)^{2}\right|}{b^{2}x^{2}\left(2+4\alpha\right)}\\\\
			\dfrac{\left|bx\right|^{3}\left|\nu - 1\right|}{\left|[\left(1+\alpha\right)b-p(1+\alpha)^{2}]bx^{2}-qa(1+\alpha)^{2}\right|} &\\
			\qquad\hbox{if}\qquad
			\left|\nu - 1\right|\geq \dfrac{\left|[\left(1+\alpha\right)b-p(1+\alpha)^{2}]bx^{2}-qa(1+\alpha)^{2}\right|}{b^{2}x^{2}\left(2+4\alpha\right)}.
		\end{array}		
		\right.
	\end{eqnarray*}
\end{theorem}

\begin{proof}
	Let $f\in\mathcal{M}_{\Sigma}(\alpha,\; x)$ be given by Taylor-Maclaurin expansion \eqref{Int-e1}. Then, there are analytic functions $u$ and $v$ such that 
	\[
	u(0)=0;\quad v(0)=0,\quad \left|u(z)\right|<1 \quad \hbox{and} \quad \left|v(z)\right|<1 \quad (\forall\; z,\; w \in \Delta),
	\]
	we can write
	\begin{equation}  \label{2.2}
	(1-\alpha)\dfrac{zf'(z)}{f(z)}+
	\alpha\left(1+\dfrac{zf''(z)}{f'(z)}\right)
	= \Pi(x,\; u(z))+1-a
	\end{equation}
	and 
	\begin{equation}  \label{2.3}
	(1-\alpha)\dfrac{wg'(w)}{g(w)}+
	\alpha\left(1+\dfrac{wg''(w)}{g'(w)}\right)=\Pi(x,\; v(w))+1-a.
	\end{equation}
	Or, equivalently,
	\begin{eqnarray}  \label{2.2a}
	&&(1-\alpha)\dfrac{zf'(z)}{f(z)}+
	\alpha\left(1+\dfrac{zf''(z)}{f'(z)}\right) \nonumber\\
	&& \qquad\qquad = 1+h_{1}(x)-a+h_{2}(x)u(z)+h_{3}(x)[u(z)]^{2}+\ldots
	\end{eqnarray}
	and 
	\begin{eqnarray}  \label{2.3a}
	&&(1-\alpha)\dfrac{wg'(w)}{g(w)}+
	\alpha\left(1+\dfrac{wg''(w)}{g'(w)}\right)
	\nonumber \\ && \qquad\qquad = 1+h_{1}(x)-a+h_{2}(x)v(w)+h_{3}(x)[v(w)]^{2}+\ldots .
	\end{eqnarray}
	From  \eqref{2.2a} and \eqref{2.3a}, we obtain
	\begin{eqnarray}  \label{2.2b}
	&&(1-\alpha)\dfrac{zf'(z)}{f(z)}+
	\alpha\left(1+\dfrac{zf''(z)}{f'(z)}\right) \nonumber\\
	&& \qquad\qquad = 1+h_{2}(x)u_{1}z+[h_{2}(x)u_{2}+h_{3}(x)u_{1}^{2}]z^{2}+\ldots
	\end{eqnarray}
	and 
	\begin{eqnarray}  \label{2.3b}
	&&(1-\alpha)\dfrac{wg'(w)}{g(w)}+
	\alpha\left(1+\dfrac{wg''(w)}{g'(w)}\right)
	\nonumber \\ && \qquad\qquad = 1+h_{2}(x)v_{1}w+[h_{2}(x)v_{2}+h_{3}(x)v_{1}^{2}]w^{2}+\ldots .
	\end{eqnarray}
	It is fairly well known that 
	\[
	\left|u(z)\right| = \left|u_{1}z+u_{2}z^{2}+\ldots\right| <1 \qquad \hbox{and}\qquad 
	\left|v(z)\right| = \left|v_{1}w+v_{2}w^{2}+\ldots\right| <1,
	\]
	then
	\[
	\left|u_{k}\right| \leq 1 \qquad \hbox{and} \qquad \left|v_{k}\right| \leq 1 \qquad (k \in \mathbb{N}).
	\]
	Thus upon comparing the corresponding coefficients in \eqref{2.2b} and \eqref{2.3b}, we have 
	\begin{equation}  \label{2.6}
	\left(1+\alpha\right)a_{2}
	= h_{2}(x)u_{1}
	\end{equation}
	\begin{equation}  \label{2.7}
	2\left(1+2\alpha\right)a_{3} - \left(1+3\alpha\right)a_{2}^{2}
	= h_{2}(x)u_{2}+h_{3}(x)u_{1}^{2}
	\end{equation}
	
	\begin{equation}  \label{2.8}
	-\left(1+\alpha\right)a_{2}
	=h_{2}(x)v_{1}
	\end{equation}
	
	and 
	\begin{equation}  \label{2.9}
	\left(3+5\alpha\right)a_{2}^{2} - 2\left(1+2\alpha\right)a_{3}
	=h_{2}(x)v_{2}+h_{3}(x)v_{1}^{2}.
	\end{equation}
	
	From \eqref{2.6} and \eqref{2.8}, we can easily see that
	\begin{equation} \label{2.9a}
	u_{1}=-v_{1}
	\end{equation}
	and
	\begin{eqnarray} 
	2(1+\alpha)^{2} a_{2}^{2} &=& [h_{2}(x)]^{2}(u_{1}^{2} + v_{1}^{2})\nonumber\\
	a_{2}^2 &=& \dfrac{[h_{2}(x)]^{2}(u_{1}^{2} + v_{1}^{2})}{2(1+\alpha)^{2}}~.\label{2.10}
	\end{eqnarray}
	If we add   \eqref{2.7} to \eqref{2.9}, we get
	\begin{eqnarray}
	2\left(1+\alpha\right) a_{2}^{2}
	= h_{2}(x)(u_{2}+v_{2}) + h_{3}(x)(u_{1}^{2}+v_{1}^{2}).\label{2.11}
	\end{eqnarray}
	By substituting  \eqref{2.10} in \eqref{2.11}, we reduce that
	\begin{eqnarray}
	a_{2}^{2} &=& \dfrac{[h_{2}(x)]^{3}\left(u_{2}+v_{2}\right)}{2\left(1+\alpha\right)[h_{2}(x)]^{2}-2h_{3}(x)(1+\alpha)^{2}}\label{2.12}
	\end{eqnarray}
	which yields
	\begin{eqnarray}
	\left|a_{2}\right| &\leq& \dfrac{\left|bx\right|\sqrt{\left|bx\right|}}{\sqrt{\left|[\left(1+\alpha\right)b-p(1+\alpha)^{2}]bx^{2}-qa(1+\alpha)^{2}\right|}}.\label{2.13}
	\end{eqnarray}
	By subtracting \eqref{2.9} from \eqref{2.7} and in view of \eqref{2.9a} , we obtain
	\begin{eqnarray}
	4(1+2\alpha)a_{3} - 4(1+2\alpha) a_{2}^{2} &=& h_{2}(x)\left(u_{2}-v_{2}\right) + h_{3}(x)\left(u_{1}^{2}-v_{1}^{2}\right)\nonumber\\
	a_{3}&=& \dfrac{h_{2}(x)\left(u_{2}-v_{2}\right)}{4(1+2\alpha)}  + a_{2}^{2}.\label{2.14}
	\end{eqnarray}
	Then in view of \eqref{2.10}, \eqref{2.14} becomes
	\begin{eqnarray*}
		a_{3}&=& \dfrac{h_{2}(x)\left(u_{2}-v_{2}\right)}{4(1+2\alpha)}  + \dfrac{[h_{2}(x)]^{2}(u_{1}^{2} + v_{1}^{2})}{2(1+\alpha)^{2}}.\label{2.15}
	\end{eqnarray*}
	Applying \eqref{HP-ab}, we deduce that 
	\begin{eqnarray*}
		\left|a_{3}\right|&\leq& \dfrac{\left|bx\right|}{2+4\alpha}  + \dfrac{b^{2}x^{2}}{(1+\alpha)^{2}}.\label{2.16}
	\end{eqnarray*}
	From \eqref{2.14}, for $\nu\in \mathbb{R},$ we write
	\begin{eqnarray}\label{Th1-Fekete-e1}
	a_{3} -\nu a_{2}^{2} = \dfrac{h_{2}(x)\left(u_{2}-v_{2}\right)}{4(1+2\alpha)} + \left(1-\nu\right)a_{2}^{2}.
	\end{eqnarray}
	By substituting \eqref{2.12} in \eqref{Th1-Fekete-e1}, we have 
	\begin{eqnarray}\label{Th1-Fekete-e2}
	a_{3} -\nu a_{2}^{2} 
	&=& \dfrac{h_{2}(x)\left(u_{2}-v_{2}\right)}{4(1+2\alpha)}+  \left(\dfrac{\left(1-\nu\right)[h_{2}(x)]^{3}\left(u_{2}+v_{2}\right)}{2 \left(1+\alpha\right)[h_{2}(x)]^{2}-2h_{3}(x)(1+\alpha)^{2}}\right)\nonumber\\
	&=& h_{2}(x)\left\{\left(\Omega(\nu,\; x) + \dfrac{1}{4\left(1+2\alpha\right)}\right)u_{2}
	+ \left(\Omega(\nu,\; x) - \dfrac{1}{4\left(1+2\alpha\right)}\right)v_{2}\right\},\nonumber\\
	\end{eqnarray}
	where
	\begin{eqnarray*}
		\Omega(\nu,\; x) = \dfrac{\left(1-\nu\right)[h_{2}(x)]^{2}}{2\left(1+\alpha\right)[h_{2}(x)]^{2}-2h_{3}(x)(1+\alpha)^{2}}.
	\end{eqnarray*}
	Hence, in view of \eqref{HP-ab}, we conclude that 
	\begin{eqnarray*}
		\left|a_{3}-\nu a_{2}^{2}\right|
		\leq \left\{
		\begin{array}{ll}
			\dfrac{\left|h_{2}(x)\right|}{2+4\alpha} &; 
			0 \leq \left|\Omega(\nu,\; x)\right|\leq \dfrac{1}{4\left(1+2\alpha\right)}\\\\
			2\left|h_{2}(x)\right|\left|\Omega(\nu,\; x)\right| &; 
			\left|\Omega(\nu,\; x)\right|\geq \dfrac{1}{4\left(1+2\alpha\right)}
		\end{array}		
		\right.
	\end{eqnarray*}
	which evidently completes the proof of Theorem \ref{Th1}.
\end{proof}
For $\alpha=1,$ Theorem \ref{Th1} readily yields the following coefficient estimates for $\mathcal{K}_{\Sigma}(x).$

\begin{corollary}
	\label{cor2.2} Let $f(z)=z+\sum\limits_{n=2}^{\infty}a_{n}z^{n}$ be in the
	class $\mathcal{K}_{\Sigma}(x)$. Then 
	\begin{align*}
	\left|a_{2}\right| &\leq \dfrac{\left|bx\right|\sqrt{\left|bx\right|}}{\sqrt{\left|[2b-4p]bx^{2}-4qa\right|}},
	\qquad \hbox{and} \qquad
	\left|a_{3}\right|\leq \dfrac{\left|bx\right|}{6} + \dfrac{b^{2}x^{2}}{4}
	\end{align*}
	and for $\nu \in \mathbb{R}$
	\begin{eqnarray*}
		\left|a_{3}-\nu a_{2}^{2}\right|
		\leq \left\{
		\begin{array}{ll}
			\dfrac{\left|bx\right|}{6} &
			~\hbox{if}~\left|\nu - 1\right|\leq \dfrac{\left|[2b-4p]bx^{2}-4qa\right|}{6b^{2}x^{2}}\\\\
			\dfrac{\left|bx\right|^{3}\left|\nu - 1\right|}{\left|[2b-4p]bx^{2}-4qa\right|} &
			~\hbox{if}~	\left|\nu - 1\right|\geq \dfrac{\left|[2b-4p]bx^{2}-4qa\right|}{6b^{2}x^{2}}.
		\end{array}		
		\right.
	\end{eqnarray*}
\end{corollary}
In view of Remark \ref{Rem-Che-P}, Theorem \ref{Th1} can be shown to yield
the following result.
\begin{corollary}
	\label{Th1-Cor} Let $f(z)=z+\sum\limits_{n=2}^{\infty}a_{n}z^{n}$ be in the
	class $\mathcal{M}_{\Sigma}(\alpha,\; t)$. Then 
	\begin{align*}
	\left|a_{2}\right| &\leq \dfrac{2t\sqrt{2t}}{\sqrt{\left|(1+\alpha)^{2}-4\alpha(1+\alpha)t^{2}\right|}},
	\qquad \hbox{and} \qquad
	\left|a_{3}\right|\leq \dfrac{t}{1+2\alpha}  + \dfrac{4t^{2}}{(1+\alpha)^{2}}
	\end{align*}
	and for $\nu \in \mathbb{R}$
	\begin{eqnarray*}
		\left|a_{3}-\nu a_{2}^{2}\right|
		\leq \left\{
		\begin{array}{ll}
			\dfrac{t}{1+2\alpha} &~\hbox{if}~
			\left|\nu - 1\right|\leq \dfrac{\left|(1+\alpha)^{2}-4\alpha(1+\alpha)t^{2}\right|}{8t^{2}\left(1+2\alpha\right)}\\\\
			\dfrac{8t^{3}\left|\nu - 1\right|}{\left|(1+\alpha)^{2}-4\alpha(1+\alpha)t^{2}\right|} &~\hbox{if}~
			\left|\nu - 1\right|\geq \dfrac{\left|(1+\alpha)^{2}-4\alpha(1+\alpha)t^{2}\right|}{8t^{2}\left(1+2\alpha\right)}.
		\end{array}		
		\right.
	\end{eqnarray*}
\end{corollary}
In view of Remark \ref{Rem-Che-P}, Theorem \ref{cor2.2} can be shown to yield
the following result.
\begin{corollary}
	\label{cor2.2-Cor} Let $f(z)=z+\sum\limits_{n=2}^{\infty}a_{n}z^{n}$ be in the
	class $\mathcal{K}_{\Sigma}(t)$. Then 
	\begin{align*}
	\left|a_{2}\right| &\leq \dfrac{t\sqrt{2t}}{\sqrt{\left|1-2t^{2}\right|}},
	\qquad \hbox{and} \qquad
	\left|a_{3}\right|\leq \dfrac{t}{3} + t^{2}
	\end{align*}
	and for $\nu \in \mathbb{R}$
	\begin{eqnarray*}
		\left|a_{3}-\nu a_{2}^{2}\right|
		\leq \left\{
		\begin{array}{ll}
			\dfrac{t}{3} &
			~\hbox{if}~\left|\nu - 1\right|\leq \dfrac{\left|1-2t^{2}\right|}{6t^{2}}\\\\
			\dfrac{2t^{3}\left|\nu - 1\right|}{\left|1-2t^{2}\right|} &
			~\hbox{if}~	\left|\nu - 1\right|\geq \dfrac{\left|1-2t^{2}\right|}{6t^{2}}.
		\end{array}		
		\right.
	\end{eqnarray*}
\end{corollary}
\begin{remark}\label{Rem-MC-Re}
	The results obtained in Corollary \ref{Th1-Cor} and Corollary \ref{cor2.2-Cor} are coincide with results of Alt\i nkaya and Yal\c cin \cite{SA-SY-2017-arXiv}.	
\end{remark}
	Next, a function $f\in \Sigma$ of the form \eqref{Int-e1} belongs to the class $\mathcal{L}_{\Sigma}(\alpha,\; x),$  $0\leqq\lambda\leqq 1,$  and $z,\; w \in \Delta,$ if the following conditions are satisfied: 
	\begin{equation*} \label{CR-RAP-NM-P1-e1}
	\left(\dfrac{zf'(z)}{f(z)}\right)^{\alpha}\left(1+\dfrac{zf''(z)}{f'(z)}\right)^{1-\alpha} \prec \Pi(x,\; z)+1-a
	\end{equation*}
	and for $g(w)=f^{-1}(w)$ 
	\begin{equation*}  \label{CR-RAP-NM-P1-e2}
	\left(\dfrac{wg'(w)}{g(w)}\right)^{\alpha}\left(1+\dfrac{wg''(w)}{g'(w)}\right)^{1-\alpha} \prec \Pi(x,\; w) +1-a,
	\end{equation*}
	where the real constants $a$ and $b$ are as in \eqref{HP-ab}.

This class also reduces to $\mathcal{S}_{\Sigma}^{\ast}(x)$ and $\mathcal{K}_{\Sigma}(x)$.  Further, as we have discussed in Remark \ref{Rem-Positive}, we can define many subclasses for the expression $\left(\dfrac{zf'(z)}{f(z)}\right)^{\alpha}\left(1+\dfrac{zf''(z)}{f'(z)}\right)^{1-\alpha}.$  
In view of Remark \ref{Rem-Che-P}, the bi-univalent function class  $\mathcal{L}_{\Sigma}^{\ast}(\alpha,\; x)$ would become the class $\mathcal{L}_{\Sigma}^{\ast}(\alpha,\; t)$. For functions in the class $\mathcal{L}_{\Sigma}(\alpha,\; x)$, the following coefficient estimates are obtained. 

\begin{theorem}
	\label{Th2} Let $f(z)=z+\sum\limits_{n=2}^{\infty}a_{n}z^{n}$ be in the
	class $\mathcal{L}_{\Sigma}(\alpha,\; x)$. Then 
	\begin{align*}
	\left|a_{2}\right| &\leq \dfrac{\left|bx\right|\sqrt{2\left|bx\right|}}{\sqrt{\left|[\left(\alpha^{2}-3\alpha+4\right)b-2p(2-\alpha)^{2}]bx^{2}-2qa(2-\alpha)^{2}\right|}}
	\quad \hbox{and} \quad
	\left|a_{3}\right|\leq \dfrac{\left|bx\right|}{6-4\alpha}  + \dfrac{b^{2}x^{2}}{(2-\alpha)^{2}}
	\end{align*}
	and for $\nu \in \mathbb{R}$
	\begin{eqnarray*}
		\left|a_{3}-\nu a_{2}^{2}\right|
		\leq \left\{
		\begin{array}{ll}
			\dfrac{\left|bx\right|}{6-4\alpha} &\\
			\qquad \hbox{if} \qquad 
			\left|\nu - 1\right|\leq \dfrac{\left|[\left(\alpha^{2}-3\alpha+4\right)b-2p(2-\alpha)^{2}]bx^{2}-2qa(2-\alpha)^{2}\right|}{4b^{2}x^{2}\left(3-2\alpha\right)}\\\\
			\dfrac{2\left|bx\right|^{3}\left|\nu - 1\right|}{\left|[\left(\alpha^{2}-3\alpha+4\right)b-2p(2-\alpha)^{2}]bx^{2}-2qa(2-\alpha)^{2}\right|} &
			\\	\qquad \hbox{if} \qquad  
			\left|\nu - 1\right|\geq \dfrac{\left|[\left(\alpha^{2}-3\alpha+4\right)b-2p(2-\alpha)^{2}]bx^{2}-2qa(2-\alpha)^{2}\right|}{4b^{2}x^{2}\left(3-2\alpha\right)}.
		\end{array}		
		\right.
	\end{eqnarray*}
\end{theorem}

\begin{proof}
	Let $f\in\mathcal{L}_{\Sigma}(\alpha,\; x)$ be given by Taylor-Maclaurin expansion \eqref{Int-e1}. Then, there are analytic functions $u$ and $v$ such that 
	\[
	u(0)=0;\quad v(0)=0,\quad \left|u(z)\right|<1 \quad \hbox{and} \quad \left|v(z)\right|<1 \quad (\forall\; z,\; w \in \Delta),
	\]
	we can write
	\begin{equation}  \label{3.2}
	\left(\dfrac{zf'(z)}{f(z)}\right)^{\alpha}\left(1+\dfrac{zf''(z)}{f'(z)}\right)^{1-\alpha} 
	= \Pi(x,\; u(z))+1-a
	\end{equation}
	and 
	\begin{equation}  \label{3.3}
	\left(\dfrac{wg'(w)}{g(w)}\right)^{\alpha}\left(1+\dfrac{wg''(w)}{g'(w)}\right)^{1-\alpha}
	=\Pi(x,\; v(w))+1-a.
	\end{equation}
	Or, equivalently,
	\begin{eqnarray}  \label{3.2a}
	&&\left(\dfrac{zf'(z)}{f(z)}\right)^{\alpha}\left(1+\dfrac{zf''(z)}{f'(z)}\right)^{1-\alpha}   \nonumber\\
	&& \qquad\qquad = 1+h_{1}(x)-a+h_{2}(x)u(z)+h_{3}(x)[u(z)]^{2}+\ldots
	\end{eqnarray}
	and 
	\begin{eqnarray}  \label{3.3a}
	&&\left(\dfrac{wg'(w)}{g(w)}\right)^{\alpha}\left(1+\dfrac{wg''(w)}{g'(w)}\right)^{1-\alpha}
	\nonumber \\ && \qquad\qquad = 1+h_{1}(x)-a+h_{2}(x)v(w)+h_{3}(x)[v(w)]^{2}+\ldots .
	\end{eqnarray}
	From  \eqref{3.2a} and \eqref{3.3a}, we obtain
	\begin{eqnarray}  \label{3.2b}
	\left(\dfrac{zf'(z)}{f(z)}\right)^{\alpha}\left(1+\dfrac{zf''(z)}{f'(z)}\right)^{1-\alpha} = 1+h_{2}(x)u_{1}z+[h_{2}(x)u_{2}+h_{3}(x)u_{1}^{2}]z^{2}+\ldots
	\end{eqnarray}
	and 
	\begin{eqnarray}  \label{3.3b}
	\left(\dfrac{wg'(w)}{g(w)}\right)^{\alpha}\left(1+\dfrac{wg''(w)}{g'(w)}\right)^{1-\alpha} = 1+h_{2}(x)v_{1}w+[h_{2}(x)v_{2}+h_{3}(x)v_{1}^{2}]w^{2}+\ldots .
	\end{eqnarray}
	It is fairly well known that 
	\[
	\left|u(z)\right| = \left|u_{1}z+u_{2}z^{2}+\ldots\right| <1 \qquad \hbox{and}\qquad 
	\left|v(z)\right| = \left|v_{1}w+v_{2}w^{2}+\ldots\right| <1,
	\]
	then
	\[
	\left|u_{k}\right| \leq 1 \qquad \hbox{and} \qquad \left|v_{k}\right| \leq 1 \qquad (k \in \mathbb{N}).
	\]
	Thus upon comparing the corresponding coefficients in \eqref{3.2b} and \eqref{3.3b}, we have 
	\begin{equation}  \label{3.6}
	\left(2-\alpha\right)a_{2}
	= h_{2}(x)u_{1}
	\end{equation}
	\begin{equation}  \label{3.7}
	2\left(3-2\alpha\right)a_{3} + \left[\left(\alpha-2\right)^{2}-3\left(4-3\alpha\right)\right]\dfrac{a_{2}^{2}}{2}
	= h_{2}(x)u_{2}+h_{3}(x)u_{1}^{2}
	\end{equation}
	
	\begin{equation}  \label{3.8}
	-\left(2-\alpha\right)a_{2}
	=h_{2}(x)v_{1}
	\end{equation}
	
	and 
	\begin{equation}  \label{3.9}
	\left[8\left(1-\alpha\right)+\dfrac{\alpha}{2}\left(\alpha+5\right)\right]a_{2}^{2} - 2\left(3-2\alpha\right)a_{3}
	=h_{2}(x)v_{2}+h_{3}(x)v_{1}^{2}.
	\end{equation}
	
	From \eqref{3.6} and \eqref{3.8}, we can easily see that
	\begin{equation} \label{3.9a}
	u_{1}=-v_{1}
	\end{equation}
	and
	\begin{eqnarray} 
	2(2-\alpha)^{2} a_{2}^{2} &=& [h_{2}(x)]^{2}(u_{1}^{2} + v_{1}^{2})\nonumber\\
	a_{2}^2 &=& \dfrac{[h_{2}(x)]^{2}(u_{1}^{2} + v_{1}^{2})}{2(2-\alpha)^{2}}~.\label{3.10}
	\end{eqnarray}
	If we add   \eqref{3.7} to \eqref{3.9}, we get
	\begin{eqnarray}
	\left(\alpha^{2}-3\alpha+4\right) a_{2}^{2}
	= h_{2}(x)(u_{2}+v_{2}) + h_{3}(x)(u_{1}^{2}+v_{1}^{2}).\label{3.11}
	\end{eqnarray}
	By substituting  \eqref{3.10} in \eqref{3.11}, we reduce that
	\begin{eqnarray}
	a_{2}^{2} &=& \dfrac{[h_{2}(x)]^{3}\left(u_{2}+v_{2}\right)}{\left(\alpha^{2}-3\alpha+4\right)[h_{2}(x)]^{2}-2h_{3}(x)(2-\alpha)^{2}}\label{3.12}
	\end{eqnarray}
	which yields
	\begin{eqnarray}
	\left|a_{2}\right| &\leq& \dfrac{\left|bx\right|\sqrt{2\left|bx\right|}}{\sqrt{\left|[\left(\alpha^{2}-3\alpha+4\right)b-2p(2-\alpha)^{2}]bx^{2}-2qa(2-\alpha)^{2}\right|}}.\label{3.13}
	\end{eqnarray}
	By subtracting \eqref{3.9} from \eqref{3.7} and in view of \eqref{3.9a} , we obtain
	\begin{eqnarray}
	4(3-2\alpha)a_{3} - 4(3-2\alpha) a_{2}^{2} &=& h_{2}(x)\left(u_{2}-v_{2}\right) + h_{3}(x)\left(u_{1}^{2}-v_{1}^{2}\right)\nonumber\\
	a_{3}&=& \dfrac{h_{2}(x)\left(u_{2}-v_{2}\right)}{4(3-2\alpha)}  + a_{2}^{2}.\label{3.14}
	\end{eqnarray}
	Then in view of \eqref{3.10}, \eqref{3.14} becomes
	\begin{eqnarray*}
		a_{3}&=& \dfrac{h_{2}(x)\left(u_{2}-v_{2}\right)}{4(3-2\alpha)}  + \dfrac{[h_{2}(x)]^{2}(u_{1}^{2} + v_{1}^{2})}{2(2-\alpha)^{2}}.\label{3.15}
	\end{eqnarray*}
	Applying \eqref{HP-ab}, we deduce that 
	\begin{eqnarray*}
		\left|a_{3}\right|&\leq& \dfrac{\left|bx\right|}{6-4\alpha}  + \dfrac{b^{2}x^{2}}{(2-\alpha)^{2}}.\label{3.16}
	\end{eqnarray*}
	From \eqref{3.14}, for $\nu\in \mathbb{R},$ we write
	\begin{eqnarray}\label{Th1-Fekete-e1}
	a_{3} -\nu a_{2}^{2} = \dfrac{h_{2}(x)\left(u_{2}-v_{2}\right)}{4(3-2\alpha)} + \left(1-\nu\right)a_{2}^{2}.
	\end{eqnarray}
	By substituting \eqref{3.12} in \eqref{Th1-Fekete-e1}, we have 
	\begin{eqnarray}\label{Th1-Fekete-e2}
	a_{3} -\nu a_{2}^{2} 
	&=& \dfrac{h_{2}(x)\left(u_{2}-v_{2}\right)}{4(3-2\alpha)}+  \left(\dfrac{\left(1-\nu\right)[h_{2}(x)]^{3}\left(u_{2}+v_{2}\right)}{\left(\alpha^{2}-3\alpha+4\right)[h_{2}(x)]^{2}-2h_{3}(x)(2-\alpha)^{2}}\right)\nonumber\\
	&=& h_{2}(x)\left\{\left(\Omega(\nu,\; x) + \dfrac{1}{4\left(3-2\alpha\right)}\right)u_{2}
	+ \left(\Omega(\nu,\; x) - \dfrac{1}{4\left(3-2\alpha\right)}\right)v_{2}\right\},\nonumber\\
	\end{eqnarray}
	where
	\begin{eqnarray*}
		\Omega(\nu,\; x) = \dfrac{\left(1-\nu\right)[h_{2}(x)]^{2}}{\left(\alpha^{2}-3\alpha+4\right)[h_{2}(x)]^{2}-2h_{3}(x)(2-\alpha)^{2}}.
	\end{eqnarray*}
	Hence, in view of \eqref{HP-ab}, we conclude that 
	\begin{eqnarray*}
		\left|a_{3}-\nu a_{2}^{2}\right|
		\leq \left\{
		\begin{array}{ll}
			\dfrac{\left|h_{2}(x)\right|}{6-4\alpha} &; 
			0 \leq \left|\Omega(\nu,\; x)\right|\leq \dfrac{1}{4\left(3-2\alpha\right)}\\
			2\left|h_{2}(x)\right|\left|\Omega(\nu,\; x)\right| &; 
			\left|\Omega(\nu,\; x)\right|\geq \dfrac{1}{4\left(3-2\alpha\right)}
		\end{array}		
		\right.
	\end{eqnarray*}
	which evidently completes the proof of Theorem \ref{Th1}.
\end{proof}

In view of Remark \ref{Rem-Che-P}, Theorem \ref{Th2} can be shown to yield

\begin{corollary}
	\label{Th2-Cor} Let $f(z)=z+\sum\limits_{n=2}^{\infty}a_{n}z^{n}$ be in the
	class $\mathcal{L}_{\Sigma}(\alpha,\; t)$. Then 
	\begin{align*}
	\left|a_{2}\right| &\leq \dfrac{2t\sqrt{2t}}{\sqrt{\left|(2-\alpha)^{2}-\left(\alpha^{2}-5\alpha+4\right)t^{2}\right|}}
	\quad \hbox{and} \quad
	\left|a_{3}\right|\leq \dfrac{t}{3-2\alpha}  + \dfrac{4t^{2}}{(2-\alpha)^{2}}
	\end{align*}
	and for $\nu \in \mathbb{R}$
	\begin{eqnarray*}
		\left|a_{3}-\nu a_{2}^{2}\right|
		\leq \left\{
		\begin{array}{ll}
			\dfrac{t}{3-2\alpha} &~\hbox{if}~ 
			\left|\nu - 1\right|\leq \dfrac{\left|(2-\alpha)^{2}-\left(\alpha^{2}-5\alpha+4\right)t^{2}\right|}{8t^{2}\left(3-2\alpha\right)}\\\\
			\dfrac{8t^{3}\left|\nu - 1\right|}{\left|(2-\alpha)^{2}-\left(\alpha^{2}-5\alpha+4\right)t^{2}\right|} &~\hbox{if}~  
			\left|\nu - 1\right|\geq \dfrac{\left|(2-\alpha)^{2}-\left(\alpha^{2}-5\alpha+4\right)t^{2}\right|}{8t^{2}\left(3-2\alpha\right)}.
		\end{array}		
		\right.
	\end{eqnarray*}
\end{corollary}

\end{document}